\documentclass[12pt]{amsart}
\usepackage{amssymb,amsmath,amsthm,latexsym,mathtools,dsfont}
\usepackage{mathrsfs}
\usepackage{mdwlist}
\usepackage{graphicx}
\usepackage{enumerate}
\usepackage[shortlabels]{enumitem}
\usepackage{a4wide}

\usepackage{ifthen}
\usepackage{xargs}

\makeatletter
\newcommand\@dotsep{4.5}
\def\@tocline#1#2#3#4#5#6#7{\relax
  \ifnum #1>\c@tocdepth 
  \else
    \par \addpenalty\@secpenalty\addvspace{#2}%
    \begingroup \hyphenpenalty\@M
    \@ifempty{#4}{%
      \@tempdima\csname r@tocindent\number#1\endcsname\relax
    }{%
      \@tempdima#4\relax
    }%
    \parindent\z@ \leftskip#3\relax
    \advance\leftskip\@tempdima\relax
    \rightskip\@pnumwidth plus1em \parfillskip-\@pnumwidth
    #5\leavevmode\hskip-\@tempdima #6\relax
    \leaders\hbox{$\m@th
      \mkern \@dotsep mu\hbox{.}\mkern \@dotsep mu$}\hfill
    \hbox to\@pnumwidth{\@tocpagenum{#7}}\par
    \nobreak
    \endgroup
  \fi}
\makeatother 

\let\oldtocsection=\tocsection
\let\oldtocsubsection=\tocsubsection

\renewcommand{\tocsection}[2]{\hspace{0em}\oldtocsection{#1}{#2}}
\renewcommand{\tocsubsection}[2]{\hspace{22pt}\oldtocsubsection{#1}{#2}}

\newcommandx{\yaHelper}[2][1=\empty]{%
\ifthenelse{\equal{#1}{\empty}}%
  { \ensuremath{ \scriptstyle{ #2 } } } 
  { \raisebox{ #1 }[0pt][0pt]{ \ensuremath{ \scriptstyle{ #2 } } } }  
}   
\newcommandx{\yrightarrow}[4][1=\empty, 2=\empty, 4=\empty, usedefault=@]{%
  \ifthenelse{\equal{#2}{\empty}}
  { \xrightarrow{ \protect{ \yaHelper[ #4 ]{ #3 } } } } 
  { \xrightarrow[ \protect{ \yaHelper[ #2 ]{ #1 } } ]{ \protect{ \yaHelper[ #4 ]{ #3 } } } } 
}

\input xy
\xyoption{all}
\usepackage{physics}

\usepackage{xcolor}
\definecolor{darkgreen}{RGB}{0, 153, 51}
\definecolor{violet}{RGB}{112, 73, 170}
\definecolor{darkred}{RGB}{153, 0, 0}

\usepackage{tikz}
\usetikzlibrary{cd}
\usetikzlibrary{calc}

\newcommand{\R}{\mathbb{R}}

\newcommand{\N}{\mathbb{N}}
\newcommand{\Z}{\mathbb{Z}}
\newcommand{\C}{\mathbb{C}}

\newcommand{\e}{\varepsilon}

\newcommand{\n}[1]{\|#1\|}
\newcommand{\nn}[1]{{\vert\kern-0.25ex\vert\kern-0.25ex\vert #1 
    \vert\kern-0.25ex\vert\kern-0.25ex\vert}}
\newcommand{\lnn}[1]{{\left\vert\kern-0.25ex\left\vert\kern-0.25ex\left\vert #1 
    \right\vert\kern-0.25ex\right\vert\kern-0.25ex\right\vert}}

\newcommand{\dist}{\mathrm{dist}}

\renewcommand{\leq}{\leqslant}
\renewcommand{\geq}{\geqslant}

\newcommand{\cs}{\mathrm{C}^\ast}

\newcommand{\Hom}{\mathrm{Hom}}

\newcommand{\usim}{\,\raise.17ex\hbox{$\scriptstyle\mathtt{\sim}$}}

\newtheorem{theorem}{Theorem}
\newtheorem{lemma}[theorem]{Lemma}
\newtheorem{proposition}[theorem]{Proposition}
\newtheorem{corollary}[theorem]{Corollary}

\theoremstyle{definition}

\theoremstyle{remark}

\numberwithin{equation}{section}

\title[Ideal C$^\ast$-completions and amenability]{A note on ideal C$^\ast$-completions and amenability}

\author{Tomasz Kochanek}
\address{Institute of Mathematics, University of Warsaw, Banacha~2, 02-097 Warsaw, Poland}
\email{tkoch@mimuw.edu.pl}

\thanks{This work has been supported by the National Science Centre grant no. 2020/37/B/ST1/01052.}

\begin{document}
\maketitle

\begin{abstract}
For a discrete group $G$, we consider certain ideals $\mathcal{I}\subset c_0(G)$ of sequences with prescribed rate of convergence to zero. We show that the equality between the full group C$^\ast$-algebra of $G$ and the C$^\ast$-completion $\cs_{\mathcal{I}}(G)$ in the sense of Brown and Guentner \cite{brown_guentner} implies that $G$ is amenable.
\end{abstract}
\maketitle

\section{Introduction} 
By a classical result of Hulanicki \cite{hulanicki}, amenable groups can be characterized by the fact that their full and reduced group C$^\ast$-algebras coincide. In \cite{brown_guentner}, Brown and Guentner obtained several far reaching generalizations of this fact by introducing a~new C$^\ast$-completion of any discrete group $G$ induced by an~algebraic ideal $\mathcal{D}$ of $\ell_\infty(G)$. Namely, the corresponding group C$^\ast$-algebra, denoted by $\cs_{\mathcal{D}}(G)$, is the completion of the group ring $\C[G]$ with respect to the norm 
$$
\n{x}_{\mathcal{D}}=\sup\big\{\n{\pi(x)}\colon \pi\mbox{ is a }\mathcal{D}\mbox{-representation}\big\},
$$
where by a $\mathcal{D}$-{\it representation} we mean a~unitary representation $\pi$ of $G$ on a~Hilbert space $H$ such that the matrix coefficient functions $\pi_{\xi,\eta}$ belong to $\mathcal{D}$ for all $\xi,\eta$ from a~dense subspace of $H$. Using this idea, Brown and Guentner provided new C$^\ast$-algebraic characterizations of a-T-menability and Kazhdan's property (T) and, among other things, they showed that the equality $\cs_{\ell^p}(G)=\cs(G)$ is equivalent to $G$ being amenable.

In this note, we consider the ideals of $c_0(G)$ consisting of sequences with prescribed rate of convergence. Namely, for $f\in c_0(G)$ and $\e>0$, we set
$$
\nu(f,\e)=\#\bigl\{s\in G\colon\abs{f(s)}\geq\e\bigr\},
$$
and define
$$
\mathcal{I}_{(a_n)}=\big\{f\in c_0(G)\colon\nu(f,\tfrac{1}{n})=O(a_n)\big\}.
$$
We show that the condition
$$
\cs_{\mathcal{I}_{(a_n)}}(G)=\cs(G)\leqno(*)
$$
is equivalent to (or implies) amenability, provided that $(a_n)$ does not grow too fast.

Amenability is strictly connected to the famous and widely studied stability property arising from a~problem posed by Ulam \cite{ulam} whether any quasimorphism can be uniformly approximated by a~homomorphism, the problem first solved for commutative groups by Hyers \cite{hyers}. We say that a~group $G$ has the {\it Hyers--Ulam property} provided that for every map $\phi\colon G\to\R$ satisfying $$\sup\bigl\{\abs{\phi(xy)-\phi(x)-\phi(y)}\colon x,y\in G\bigr\}<\infty$$ we have $\dist(\phi,\Hom(G,\R))<\infty$. It is known (see \cite{szekelyhidi}) that every amenable group has the Hyers--Ulam property, but the converse is not true which is witnessed e.g. by the groups $\mathrm{SL}(n,\Z)$ for $n\geq 3$. Although there is an~algebraic characterization of the Hyers--Ulam property, due to Bavard \cite{bavard}, no C$^\ast$-algebraic characterization is known.

Hence, a natural question concerning Ulam stability reads as follows: Is there an increasing sequence $(a_n)\subset\R_+$ such that for any discrete group $G$ the following characterization holds true: $G$ has the Hyers--Ulam property if and only if condition ($\ast$) holds true? Our result reduces the size of the set of possible candidates for $(a_n)$.

\section{Results}
In what follows, $G$ stands for a general discrete group. We will need the following two results proved by Brown and Guentner.
\begin{proposition}[{see \cite[Remark~2.5]{brown_guentner}}]\label{BG_P}
For any ideal $\mathcal{D}\subset \ell^\infty(G)$, $\cs_{\mathcal{D}}(G)$ has a~faithful $\mathcal{D}$-representation.
\end{proposition}

\begin{theorem}[{\cite[Thm.~3.2]{brown_guentner}}]\label{BG_T}
Let $\mathcal{D}\subset\ell^\infty(G)$ be a~translation-invariant ideal. Then, we have $\cs_{\mathcal{D}}(G)=\cs(G)$ if and only if there exists a~sequence $(h_n)\subset\mathcal{D}$ of positive-definite functions converging pointwise to the constant one function.
\end{theorem}

Our main result reads as follows.

\begin{theorem}\label{thm1}
{\bf (a) }If $(a_n)=O(n^{2+\delta})$ for every $\delta>0$, then 
$$
\cs_{\mathcal{I}_{(a_n)}}(G)=\cs_r(G)
$$
and hence condition {\rm ($\ast$)} is equivalent to $G$ being amenable.

\vspace*{2mm}\noindent
{\bf (b) }Suppose a sequence $(a_n)\subset\R_+$ is such that for some $k>0$, we have
$$
\sum_{n=1}^\infty\frac{a_n}{n^k}<\infty .\leqno(**)
$$

\vspace*{1mm}\noindent
Then, condition {\rm ($\ast$)} implies that $G$ is amenable.
\end{theorem}

\begin{lemma}\label{lem}
Let $f\in c_0(G)$ and $p\geq 1$. We have $f\in \ell^{p}(G)$ if and only if the series 
\begin{equation}\label{series}
\sum_{n=1}^\infty\nu(f,\tfrac{1}{n})n^{-(p+1)}
\end{equation}
converges.
\end{lemma}
\begin{proof}
Let $\Gamma_n=\{s\in G\colon \tfrac{1}{n}\leq\abs{f(s)}<\tfrac{1}{n-1}\}$ for $n\in\N$ with the convention $\tfrac{1}{0}=\infty$, and note that 
$$
\sum_{s\in G}\abs{f(s)}^p=\sum_{n=1}^\infty\sum_{s\in\Gamma_n}\abs{f(s)}^p.
$$
Since $\abs{\Gamma_1}=\nu(f,1)$ and $\abs{\Gamma_n}=\nu(f,\tfrac{1}{n})-\nu(f,\tfrac{1}{n-1})$ for $n\geq 2$, we have
\begin{equation}\label{est1}
\sum_{s\in G}\abs{f(s)}^p\leq \nu(f,1)\cdot \n{f}_\infty+\sum_{n=1}^\infty \big(\nu(f,\tfrac{1}{n+1})-\nu(f,\tfrac{1}{n})\big)\cdot n^{-p}
\end{equation}
and
\begin{equation}\label{est2}
\sum_{s\in G}\abs{f(s)}^p\geq \sum_{n=1}^\infty \big(\nu(f,\tfrac{1}{n+1})-\nu(f,\tfrac{1}{n})\big)\cdot (n+1)^{-p}.
\end{equation}
Denote $d_n=\nu(f,\tfrac{1}{n+1})-\nu(f,\tfrac{1}{n})$; the series occurring in \eqref{est1} is the limit of partial sums
$$
\lim_{N\to\infty}\sum_{n=1}^N d_nn^{-p}=\lim_{N\to\infty}\Bigg[\sum_{n=1}^{N-1}(d_1+\ldots+d_n)\big(n^{-p}-(n+1)^{-p}\big)+(d_1+\ldots+d_N)\cdot N^{-p}\Bigg].
$$
Since $d_1+\ldots+d_N=\nu(f,\tfrac{1}{N+1})-\nu(f,1)$, the above limit exists if and only if the series 
$$
\sum_{n=1}^\infty \nu(f,\tfrac{1}{n+1})(n^{-p}-(n+1)^{-p})
$$
converges. By Lagrange's mean value theorem, we have $n^{-p}-(n+1)^{-p}=p\,\theta_n^{-(p+1)}$ for some $\theta_n\in (n,n+1)$, hence the above series converges if and only if \eqref{series} converges. 

We have proved that the convergence of series \eqref{series} implies that $f\in\ell^p(G)$. The converse implication is proved in a~similar fashion by using estimate \eqref{est2} instead of \eqref{est1}.
\end{proof}

\begin{proof}[Proof of Theorem~\ref{thm1}]
{\bf (a) }Suppose that $(a_n)=O(n^{2+\delta})$ for each $\delta>0$. Then for any $f\in \mathcal{I}_{(a_n)}$ and any $\delta>0$ there is $C_\delta>0$ such that 
$$
\nu(f,\tfrac{1}{n}) n^{-(p+1)}\leq C_\delta\cdot n^{-p+1+\delta}\quad (n\in\N).
$$
Therefore, series \eqref{series} converges for every $p>2$ and hence Lemma~\ref{lem} implies that 
\begin{equation}\label{iinc}
\mathcal{I}_{(a_n)}\subseteq \bigcap_{\e>0}\ell^{2+\e}(G).
\end{equation}

By the Cowling--Haagerup--Howe theorem \cite{CHH},  if $\pi\colon G\to\mathcal{B}(H)$ is a~unitary representation of $G$ with a~cyclic vector $v\in H$ such that $\pi_{v,v}\in\bigcap_{\e>0}\ell^{2+\e}(G)$, then $\pi$ is weakly contained in the regular representation $\lambda$, i.e. $\n{\pi(x)}\leq \n{\lambda(x)}$ for each $x\in G$.

\vspace*{1mm}
Now, for any fixed $x\in\cs_{\mathcal{I}_{(a_n)}}(G)$ we use Proposition~\ref{BG_P} to \vspace*{-1pt}pick a~cyclic $\mathcal{I}_{(a_n)}$-representation $\pi$ with $\pi(x)\neq 0$ (the restriction of a~faithful $\mathcal{I}_{(a_n)}$-representation to a~cyclic subspace). Then, as explained above, inclusion \eqref{iinc} implies that $\pi$ is weakly contained in the regular representation. Therefore, $x$ is not in the kernel of the canonical \vspace*{-1pt}map $\cs_{\mathcal{I}_{(a_n)}}(G)\to \cs_r(G)$, which proves that $\cs_{\mathcal{I}_{(a_n)}}(G)=\cs_r(G)$.

\vspace*{2mm}\noindent
{\bf (b) }This is essentially \cite[Remark~2.13]{brown_guentner} by Brown and Guentner. Notice that condition ($\ast\ast$) says that for any $f\in \mathcal{I}_{(a_n)}$ we have $f^k\in \ell^1(G)$. Indeed, let $C>0$ be such that $\nu(f,\tfrac{1}{n})\leq Ca_n$. Then, the inequality $\abs{f(x)}^k\geq n^{-k}$ holds true for at most $Ca_n$ elements $x\in G$, hence $\n{f^k}_1\leq \nu (f,1)\cdot \n{f}_\infty+C\sum_{n\geq 2}a_n n^{-k}<\infty$.

\vspace*{1mm}
Now, by Theorem~\ref{BG_T}, condition ($\ast$) implies that there exists a sequence $(h_n)\subset \mathcal{I}_{(a_n)}$ of positive-definite functions converging pointwise to the constant one function. In view of ($\ast\ast$), we have $(h_n^k)\subset\ell^1(G)$; if $f_n\subset c_{00}(G)$ approximates the square root of $h_n^k$ in $\cs_r(G)$, then $h_n^k$ is approximated by the finitely supported positive-definite functions $f_n^\ast f_n$. This yields $\cs_r(G)=\cs(G)$, i.e. $G$ is amenable.
\end{proof}

We conclude our note with a corollary which shows that if there is any ideal of the form $\mathcal{I}_{(a_n)}$ characterizing the Hyers--Ulam property for discrete groups, then $(a_n)$ must grow quite rapidly. This follows from Theorem~\ref{thm1} and the fact that the Hyers--Ulam property is weaker than amenability.

\begin{corollary}
If there exists a sequence $(a_n)\subset\R_+$ such that condition {\rm ($\ast$)} characterizes the Hyers--Ulam property, then $(a_n)$ grows faster than any polynomial.
\end{corollary}

\vspace*{2mm}\noindent
{\bf Acknowledgement. }I acknowledge with gratitude the support from the National Science Centre, grant OPUS 19, project no.~2020/37/B/ST1/01052.

\bibliographystyle{amsplain}

\end{document}